\pdfoutput=1

\RequirePackage{snapshot}
\documentclass[12pt,fabfinal,nofabrule]{fabarticle}

\usepackage[OT1]{fontenc}

\usepackage[utf8]{inputenc}

\usepackage
[a4paper,tmargin=3truecm,bmargin=3truecm,rmargin=2.5truecm,lmargin=2.5truecm,twoside,verbose=true]{geometry}
\usepackage{amsmath,amssymb,latexsym}

\usepackage{stmaryrd,wasysym,upgreek,mathrsfs,dsfont,mathtools}
\usepackage[english]{babel}
\usepackage{graphicx,color}
\usepackage{slashed,subfig}

\usepackage{array}
\usepackage{multirow}
\usepackage{hhline}

\allowdisplaybreaks[1]

\usepackage[pdftex]{hyperref}
\hypersetup{%
  colorlinks=true,%
  linkcolor=black,%
  anchorcolor=black,%
  citecolor=black,%
  urlcolor=black,%
  pdftitle={},%
  pdfauthor={Fabien Vignes-Tourneret},%
  pdfsubject={Research paper},%
  pdfkeywords={Graph polynomial, Bollobas-Riordan, partial duality, multivariate},%
  bookmarksopen=false,%
}

\usepackage[hyperref,thmmarks,amsmath]{fabntheorem}
\usepackage{fabmesthm-en}
\usepackage{fabmacromath}

\usepackage{pdfsync}

\numberwithin{equation}{section}

\usepackage[square,numbers,comma]{natbib}

\usepackage{bibentry}

\title{\textsf{The multivariate signed \BRp}}
\author{\textrm{Fabien Vignes-Tourneret}}
\date{}

\begin{document}
\nobibliography*
\maketitle


  

\begin{abstract}
  We generalise the signed \BRp{} of [\bibentry{Chmutov2007ab}] to a multivariate signed polynomial $Z$ and study its properties. We prove the invariance of $Z$ under the recently defined partial duality of [\bibentry{Chmutov2007aa}] and show that the duality transformation of the multivariate Tutte polynomial is a direct consequence of it.
\end{abstract}
\vspace*{.5cm}


\section{Introduction}
\label{sec:introduction}

Ribbon graphs are surfaces with boundary together with a decomposition into a union of closed topological discs of two types, edges and vertices. These sets are subject to some natural axioms recalled in section \ref{sec:ribbon-graphs}. For such a generalisation of the usual graphs, B.~Bollob\'as and O.~Riordan found a topological version of the \Tp{} \cite{Bollobas2001aa,Bollobas2002aa}. In the following, we will refer to this generalisation as the \BRp.\\

S.~Chmutov and I.~Pak introduced a first generalisation of the \BRp{} in \cite{Chmutov2007ab}. It is a three-variable polynomial $R_{s}$ defined on \emph{signed} ribbon graphs. Recall that a graph is said to be signed if to each of its edges an element of $\{+,-\}$ is assigned. Then S.~Chmutov defined \cite{Chmutov2007aa} a new kind of duality with respect to a spanning subgraph\footnote{Considering mainly ribbon graphs, we will write \emph{subgraph} instead of \emph{subribbon graph}. We hope that it will not lead to any misunderstanding.} of a ribbon graph (see section \ref{sec:partial-duality} for a definition). This allows him to prove that the Kauffman bracket of a virtual link diagram $L$ equals the signed \BRp{} of a ribbon graph $G_{L}^{\ks}$, see \eqref{eq:AlternatBRCorresp}. The latter is constructed from a state $\ks$ of $L$.
\begin{align}
  [L](A,B,d)=A^{n(G_{L})}B^{r(G_{L})}d^{k(G_{L})-1}R_{s}(G^{\ks}_{L};\frac{Ad}{B}+1,\frac{Bd}{A},\frac 1d).\label{eq:AlternatBRCorresp}
\end{align}
The new \emph{partial} duality of S.~Chmutov ensures the independence of the right-hand side of \eqref{eq:AlternatBRCorresp} with respect to the state $\ks$.\\

Recall that there exists a natural notion of duality for ribbon graphs. Given such a graph $G$, its dual $G^{\star}$ is built as follows. First glue a disc along each boundary component of $G$. Then remove the interior of each vertex-disc of $G$. The vertex-discs of $G^{\star}$ are the glued discs and its edge-discs are the same as $G$. In the sequel we will refer to this duality as the \emph{natural} duality. The natural duality is a special case of Chmutov's duality. 

In \cite{Chmutov2007aa}, S.~Chmutov also studied the properties of the signed \BRp{} he defined with I.~Pak under the \emph{partial} duality. At the end of this article, he asked whether his work can be generalised to a multivariate polynomial (by multivariate we mean that to each edge corresponds a different variable). It is indeed a natural question to ask. Generally, multivariate generalisation of graph invariant polynomials encode more information than their univariate counterpart. Moreover they are usually easier to handle, see \cite{Ellis-Monaghan2008aa,Ellis-Monaghan2008ab,Sokal2005aa} for review and examples. This article is an answer to Chmutov's question.\\

After briefly reviewing the notions of ribbon graphs and partial duality in section \ref{sec:part-dual-ribb}, the section \ref{sec:generalised-duality} is devoted to the definition and first properties of our multivariate signed \BRp{} $Z$. We derive, there, its behaviour under disjoint union and one-point join as well as its contraction-deletion reduction relations. In section \ref{sec:spann-tree-expans} we give two alternative definitions of the polynomial: namely, a spanning tree and a quasi-tree expansion. The former is very much in the spirit of the spanning tree expansion of Bollob\'as and Riordan in \cite{Bollobas2002aa} and of the one of Kauffman \cite{Kauffman1989aa} for the signed Tutte polynomial. The latter relies on the work of A.~Champanerkar, I.~Kofman and N.~Stoltzfus \cite{Champanerkar2007aa}. Our main theorem, namely the invariance of $Z$ under partial duality, is stated and proved in section \ref{sec:generalised-duality-1}. We can then extend the contraction-deletion relations but only on the surface $xyz^{2}\fide q=1$. Finally we prove that the (natural) duality transformation of the multivariate Tutte polynomial (see proposition \ref{prop:MultiTutteDuality} and \cite{Sokal2005aa}) is a direct consequence of the partial duality transformation of our multivariate signed polynomial. 

\vspace{1cm}
\paragraph{Acknolewdgements}
\label{sec:acknolewdgements}

I am grateful to V.~Rivasseau for having introduced me to the subject. I would also like to warmly thank S.~Chmutov. We had very interesting and fruitful discussions. He also pointed out to me that the change of the sign function (corollary \ref{cor:ChangeSignFct}) can be generalised to a flip of a single edge sign (see proposition \ref{prop:FlipSign}). He also read, carefully, preliminary versions of the manuscript.

I also thank the Hausdorff Institute in Bonn (Germany) during a visit of which this work was initiated.

Finally, I thank the anonymous referees who made interesting suggestions which led to this improved version.

\newpage
\section{Partial duality of a ribbon graph}
\label{sec:part-dual-ribb}

\subsection{Ribbon graphs}
\label{sec:ribbon-graphs}

A ribbon graph $G$ is a (not necessarily orientable) surface with boundary represented as the union of two sets of closed topological discs called vertices $V(G)$ and edges $E(G)$. These sets satisfy the following:
\begin{itemize}
  \item vertices and edges intersect by disjoint line segment,
  \item each such line segment lies on the boundary of precisely one vertex and one edge,
  \item every edge contains exactly two such line segments.
\end{itemize}
Figure \ref{RibbonEx1} shows an example of a ribbon graph. Note that we allow the edges to twist (giving the possibility to the surfaces associated to the ribbon graphs to be non-orientable). A priori, an edge may twist more than once but the \BRp{} only depends on the parity of the number of twists (this is indeed the relevant information to count the boundary components of a ribbon graph) so that we will only consider edges with at most one twist.
\begin{figure}[htb]
  \centering
  \subfloat[A signed ribbon graph]{{\label{RibbonEx1}}\includegraphics[scale=.8]{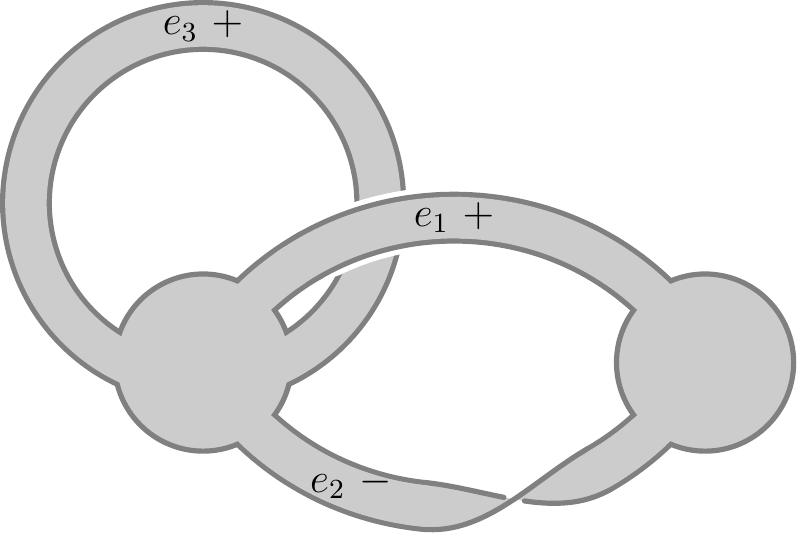}}\hspace{2cm}
  \subfloat[The combinatorial representation]{{\label{CombRep1}}\includegraphics[scale=.8]{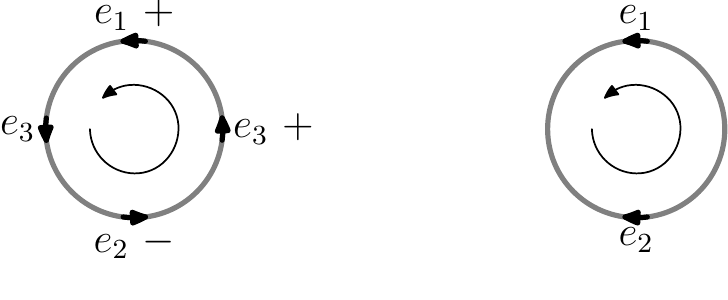}}
  \label{TwoRepEx}
  \caption{Two representations of a ribbon graph}
\end{figure}

A ribbon graph $G$ is said to be \textbf{signed} if an element of $\{+,-\}$ is assigned to each edge. This is achieved via a function $\veps_{G}:E(G)\to\{-1,1\}$.

For the construction of partial dual graphs, another (equivalent) representation of ribbon graphs will be useful. It has been introduced in \cite{Chmutov2007aa} and will be referred hereafter as the ``combinatorial representation''. It can be described as follows: for any ribbon graph $G$, pick up an orientation of each vertex-disc and each edge-disc. The orientation of the edges induces an orientation of the line segments along which they intersect the vertices. Then draw all vertex-discs as disjoint circles in the plane oriented counterclockwise (say), but for the edges, draw only the arrows corresponding to the orientation of the line segments. Figure \ref{CombRep1} gives the combinatorial representation of the graph of figure \ref{RibbonEx1}.

Given a combinatorial representation, one reconstructs the corresponding ribbon graph as follows. Each circle of the representation is filled: this gives the vertex-discs. Let us consider a couple $c_{e}$ of arrows with the same label (i.e.\@ corresponding to the same edge). These two arrows belong to the boundaries of vertices $v_{1}$ and $v_{2}$, which may be equal. One draws an edge which intersects $v_{1}$ and $v_{2}$ along the arrows of $c_{e}$. We now have to decide whether this edge twists or not. This depends on the relative direction of the two arrows. Actually there is a unique choice (twist or not) such that there exists an orientation of the edge which reproduces the couple of arrows under consideration. So we proceed as explained for each couple of arrows with a common label.

\paragraph{Loops}
\label{sec:loops}

Contrary to the graphs, the ribbon graphs may contain four different kinds of loops. A loop may be \textbf{orientable} or not, a \textbf{non-orientable} loop being a twisting edge. Let us consider the general situations of figure \ref{fig:loopRibbon}. The boxes $A$ and $B$ represent any ribbon graph so that the picture \ref{OrLoop} (resp.\@ \ref{NonOrLoop}) describes any ribbon graph $G$ with an orientable (resp.\@ non-orientable) loop $e$ at vertex $v$. A loop is called \textbf{nontrivial} if there is a path in $G$ from $A$ to $B$ which does not contain $v$. If not, the loop is called \textbf{trivial} \cite{Bollobas2002aa}.
\begin{figure}[htb]
  \centering
  \subfloat[An orientable loop]{{\label{OrLoop}}\includegraphics[scale=.8]{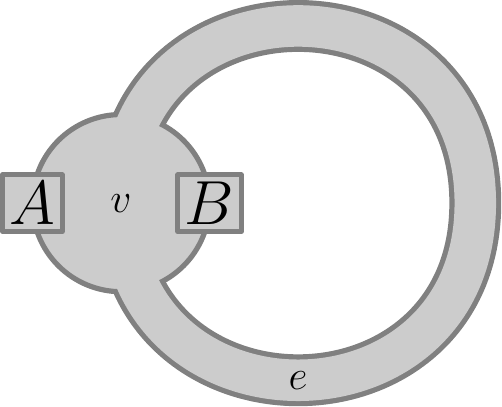}}\hspace{2cm}
  \subfloat[A non-orientable loop]{{\label{NonOrLoop}}\includegraphics[scale=.8]{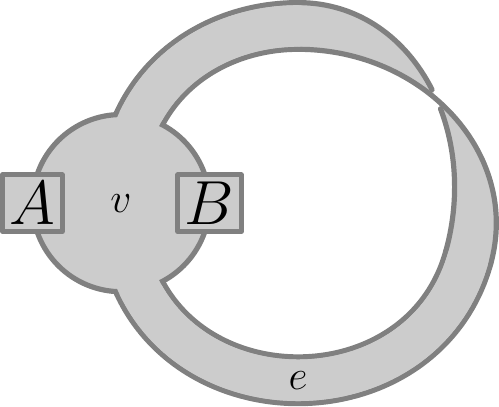}}
  \caption{Loops in ribbon graphs}
  \label{fig:loopRibbon}
\end{figure}

\subsection{Partial duality}
\label{sec:partial-duality}

S.~Chmutov introduced recently (see \cite{Chmutov2007aa}) a new ``generalised duality'' for ribbon graphs which generalises the usual notion of duality. In \cite{Moffatt2008aa}, I.~Moffatt renamed this new duality as ``partial duality''. We adopt this designation here. We now describe the construction of a partial dual graph and give a few properties of the partial duality.\\

Let $G$ be a ribbon graph and $E'\subset E(G)$. Let $F_{E'}$ be the spanning subgraph of $G$ the edge-set of which is $E'$. We will construct the dual $G^{E'}$ of $G$ with respect to the edge-set $E'$, see figure \ref{fig:ribbonEx} for an example. Recall that each edge of $G$ intersects one or two vertex-discs along two line segments. In the following, each time we write ``line segment'', we mean the intersection of an edge and a vertex.

We now construct the combinatorial representation of the partial dual $G^{E'}$ of $G$. We first choose an orientation for each edge of $G$. It induces an orientation of the boundaries of the edges. For each edge in $E(G)-E'$, and as was explained for the combinatorial representation, we draw one arrow per oriented line segment at the boundary of that edge. For the edges in $E'$, we proceed differently. Considering them as rectangles, they have two opposite sides that they share with one or two disc-vertices: these are the line segments, defined above. But they also have two other opposite sides that we call ``long sides''. The chosen orientation induces an orientation of the long sides of the edges in $E'$, see figure \ref{fig:BoundComp} for an example. We draw an arrow on each long side of each edge in $E'$ according to the chosen orientation. Now draw each boundary component of $F_{E'}$ as a circle with arrows corresponding to the edges of $G$. The result is the combinatorial representation of $G^{E'}$, see figure \ref{fig:DualExComb} and \ref{fig:DualEx}. Note that $G$ and $G^{E'}$ are generally embedded into different surfaces (they may have different genera).\\

As in the case of the natural duality, and for any $E'\subset E(G)$, there is a bijection between the edges of $G$ and the edges of its partial dual $G^{E'}$. Let $\phi:E(G)\to E(G^{E'})$ denote this bijection. We explain now how it is defined from the construction of the partial dual graph. As explained above, on each edge $e\in E(G)$, we draw two arrows compatible with an arbitrarily chosen orientation of this edge. If $e\in E'$, these arrows are drawn on the long sides of $e$. If $e\in E(G)\setminus E'$, they belong to the line segments along which $e$ intersects its end-vertices. Anyway, we label this couple of arrows with $\phi(e)$. Proceeding like that for all edges of $G$, we build the combinatorial representation of the dual $G^{E'}$ - namely we get one circle per boundary component of the spanning subgraph $F_{E'}$ of $G$. On each of these circles, there are arrows which represent the edges of $G^{E'}$. For each couple $c_{e'}$ of arrows that is for each edge $e'$ of $G^{E'}$, there exists a unique $e\in E(G)$ such that  $c_{e'}$ bears the label $\phi(e)$. The map $\phi$ is then clearly a bijection.\\

For signed graphs, the partial duality comes with a change of the sign function. The function $\veps_{G^{E'}}$ is defined by the following equations: $\forall e\in E-E',\,\veps_{G^{E'}}(e)=\veps_{G}(e)$ and $\forall e\in E',\,\veps_{G^{E'}}(e)=-\veps_{G}(e)$. For unsigned ribbon graphs and if $E'=E$, the partial duality is the usual duality which exchanges faces (boundary components) and vertices.
\begin{figure}[!htp]
  \centering
  \subfloat[A ribbon graph $G$ with $E'=\{e_{1}\}$]{{\label{fig:ribbonEx}}\includegraphics[scale=.8]{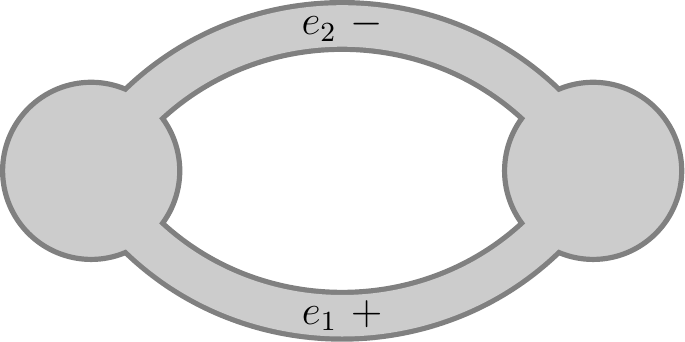}}\hspace{2cm}
  \subfloat[The combinatorial representation of $G$]{{\label{fig:ribbonExReprArrow}}\includegraphics[scale=.8]{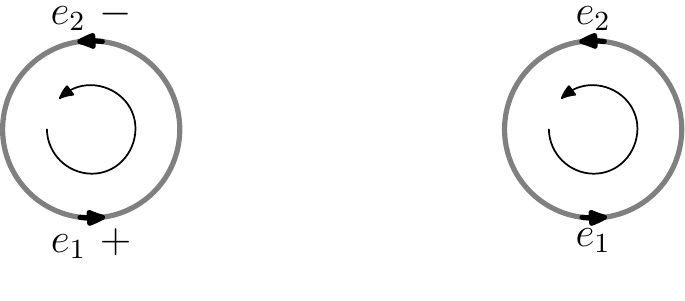}}\\
  \subfloat[The boundary component of $F_{E'}$]{{\label{fig:BoundComp}}\includegraphics[scale=.8]{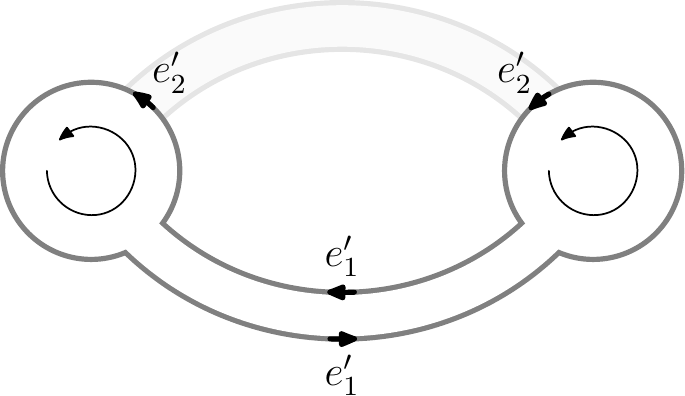}}\hspace{2cm}
  \subfloat[The combinatorial representation of $G^{E'}$]{{\label{fig:DualExComb}}\includegraphics[scale=.8]{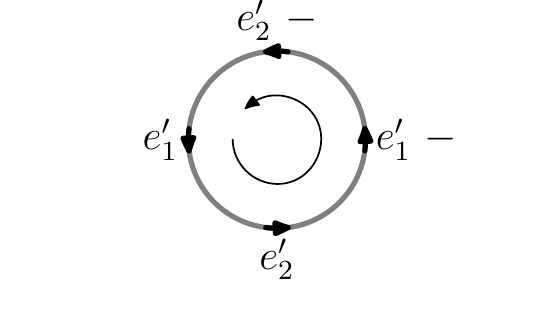}}\\
  \subfloat[The dual $G^{E'}$]{{\label{fig:DualEx}}\includegraphics[scale=.8]{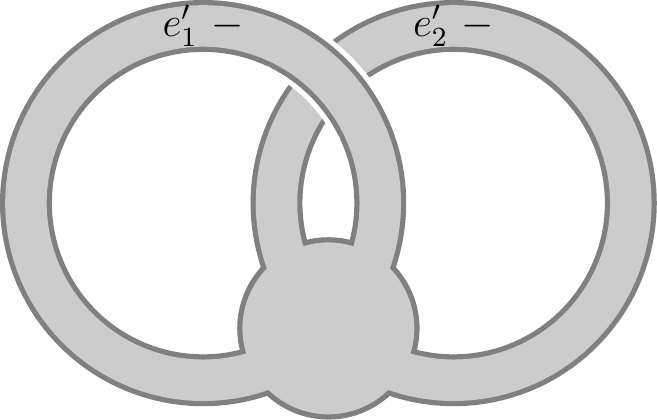}}
  \caption{Construction of a partial dual}
  \label{PartDualEx}
\end{figure}

S.~Chmutov proved among other things the following basic properties of his partial duality:
\begin{lemma}[\cite{Chmutov2007aa}]
  \label{lem:SimpleProp}
  For any ribbon graph $G$ and any subset of edges $E'\subset E(G)$, we have
  \begin{itemize}
  \item let $e\notin E'$, then $G^{E'\cup\{e\}}=(G^{E'})^{\{e\}}$,
  \item $(G^{E'})^{E'}=G$ and
  \item the partial duality preserves the number of connected components.
  \end{itemize}
\end{lemma}

The partial duality allows an interesting and fruitful definition of the contraction of an edge:
\begin{defnb}[Contraction of an edge \cite{Chmutov2007aa}]\label{def:Contraction}
  Let $G$ be a ribbon graph and $e\in E(G)$ any of its edges. We define the contraction of $e$ by:
  \begin{align}
    G/e\defi&G^{\{e\}}-e.\label{eq:ContractionDef}
  \end{align}
\end{defnb}
From the definition of the partial duality, one easily checks that, for an edge incident with two different vertices, the definition \ref{def:Contraction} coincides with the usual intuitive contraction of an edge. The contraction of a loop depends on its orientability, see figures \ref{fig:OrLoopContraction} and \ref{fig:NonOrLoopContraction}.
\begin{figure}[!htp]
  \centering
  \begin{minipage}[c]{.4\linewidth}
    \centering
    \includegraphics[scale=.8]{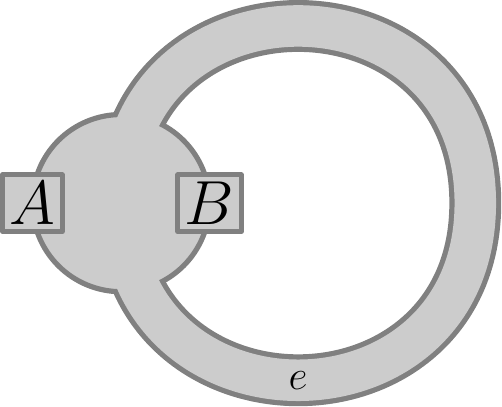}\\
    A ribbon graph $G$ with an orientable loop $e$
  \end{minipage}%
  \qquad$\longrightarrow$\qquad%
  \begin{minipage}[c]{.4\linewidth}
    \centering
    \includegraphics[scale=.8]{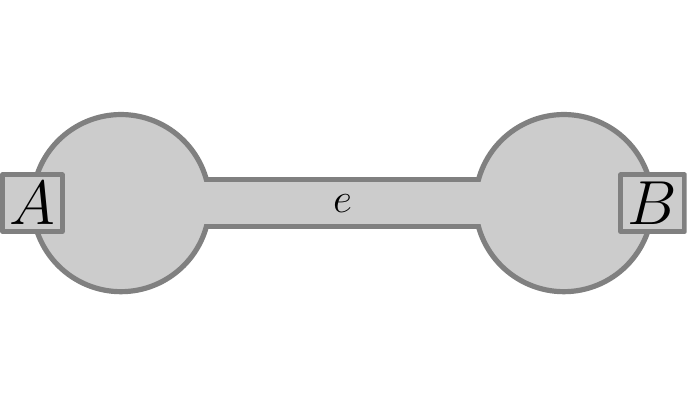}\\
    $G^{\{e\}}$\\
    \phantom{orientable loop $e$}
  \end{minipage}\\%
  $\longrightarrow$\qquad%
  \begin{minipage}[c]{.4\linewidth}
    \centering
    \includegraphics[scale=.8]{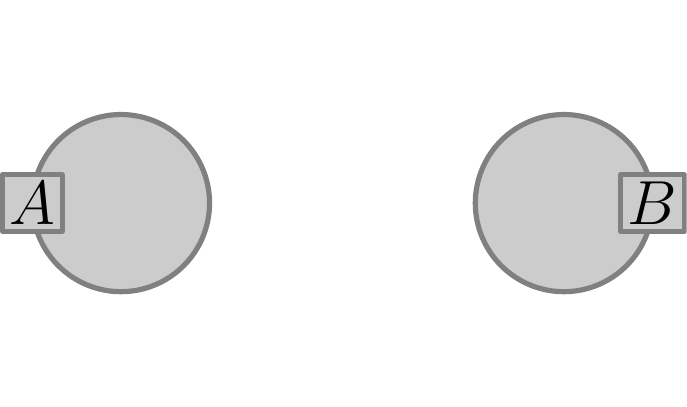}\\
    $G/e=G^{\{e\}}-e$
  \end{minipage}
  \caption{Contraction of an orientable loop}
  \label{fig:OrLoopContraction}
\end{figure}
\begin{figure}[!htp]
  \centering
  \begin{minipage}[c]{.4\linewidth}
    \centering
    \includegraphics[scale=.8]{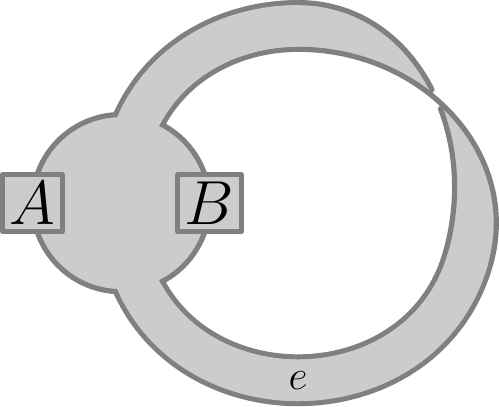}\\
    A ribbon graph $G$ with a non-orientable loop $e$
  \end{minipage}%
  \qquad$\longrightarrow$\qquad%
  \begin{minipage}[c]{.4\linewidth}
    \centering
    \includegraphics[scale=.8]{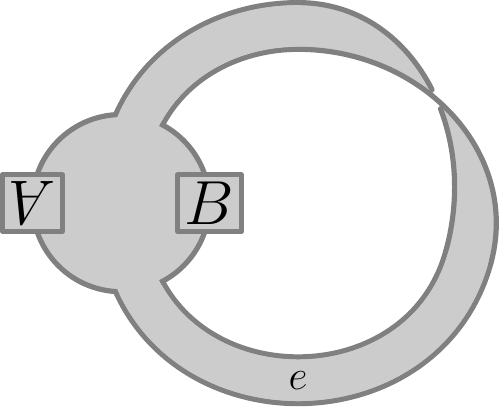}\\
    $G^{\{e\}}$\\
    \phantom{non-orientable loop $e$}
  \end{minipage}\\%
  $\longrightarrow$\qquad%
  \begin{minipage}[c]{.4\linewidth}
    \centering
    \includegraphics[scale=.8]{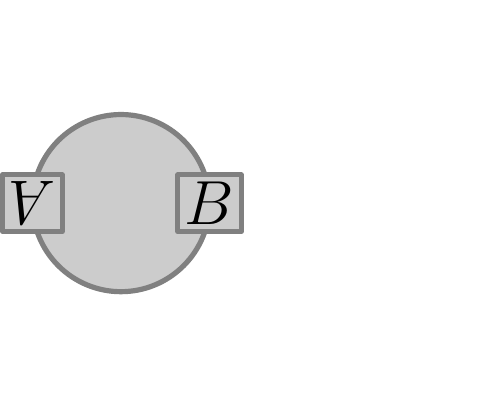}\\
    $G/e=G^{\{e\}}-e$
  \end{minipage}
  \caption{Contraction of a non-orientable loop}
  \label{fig:NonOrLoopContraction}
\end{figure}

Different definitions of the contraction of a loop have been used in the litterature. One can define $G/e\defi G-e$. In \cite{Huggett2007aa}, S.~Huggett and I.~Moffatt give a definition which leads to surfaces which are not ribbon graphs anymore. The definition \ref{def:Contraction} maintains the duality between contraction and deletion, and, as will be shown in section \ref{sec:generalised-duality-1}, it allows one to get reduction relations for nontrivial loops.

\section{Multivariate signed polynomial}
\label{sec:generalised-duality}

In this section, we define the multivariate version of the signed \BRp{} introduced in \cite{Chmutov2007ab}. We derive its behaviour under disjoint union and one-point join, and prove its contraction-deletion relations.

\subsection{Definition}
\label{sec:definition}

Let $G$ be a signed ribbon graph. Let us define $E(G)\fide E_{+}\cup E_{-}$ with $E_{\pm}$ being the set of positive (resp.\@ negative) edges of $G$. We write $e_{\pm}$ for the corresponding cardinalities. For any spanning subgraph $F=(V(G),E(F))$ of $G$, let $\Fb$ be the spanning subgraph of $G$ with edge-set $E(G)-E(F)$ and $s(F)\defi \frac 12(e_{-}(F)-e_{-}(\Fb))$.

For the rest of this article we use the following notations:
\begin{itemize}
\item $v(G)=\card V(G)$ is the number of vertices of $G$,
\item $e(G)=\card E(G)$ is the number of edges of $G$,
\item $k(G)$ its number of components,
\item $r(G)=v(G)-k(G)$ its rank,
\item $n(G)=e(G)-r(G)$ its nullity and
\item $f(G)$ its number of boundary components (faces).
\end{itemize}

Let $R_{s}(G;x+1,y,z)$ be the signed \BRp{} for ribbon graphs introduced in \cite{Chmutov2007ab}:
\begin{align}
  R_{s}(G;x+1,y,z)=&\sum_{F\subseteq G}x^{r(G)-r(F)+s(F)}y^{n(F)-s(F)}z^{k(F)-f(F)+n(F)}\label{eq:BRPolyDef}\\
  \fide&\,x^{-k(G)}(yz)^{-v(G)}Z(G;xyz^{2},yz,z),\\
  Z(G;xyz^{2},yz,z)=&\sum_{F\subseteq G}(xyz^{2})^{k(F)}(yz)^{e(F)}z^{-f(F)}x^{s(F)}y^{-s(F)}.
  \intertext{We define new variables $q\defi xyz^{2}$, $\alpha\defi yz$, $c\defi z^{-1}$ and get:}
  Z(G;q,\alpha,c)=&\sum_{F\subseteq G}q^{k(F)+s(F)}\alpha^{e(F)-2s(F)}c^{f(F)}.
\end{align}
The generalisation to the multivariate case is then obvious.
\begin{defn}
  Let $G$ be any signed ribbon graph, possibly with loops and multiple edges. Let $q,z\in\C$ and for all $e\in E(G)$, let $\alpha_{e}\in\C$. Let also $\balpha$ denote the set $\{\alpha_{e}\}_{e\in E(G)}$. We define the multivariate signed \BRp{} as follows:
  \begin{align}
    Z(G;q,\mathbf{\alpha},c)\defi&\sum_{F\subseteq G}q^{k(F)+s(F)}\Big(\prod_{\substack{e\in E_{+}(F)\\\cup E_{-}(\bar{F})}}\alpha_{e}\Big)c^{f(F)}.\label{eq:SMBRDef}
  \end{align}
\end{defn}
The multivariate polynomial $Z$ is clearly a multivariate generalisation of $R_{s}$. Indeed if for any $e\in E(G),\,\alpha_{e}=yz$ and if we let $\mathbf{yz}$ be the corresponding set, we have
\begin{align}
  R_{s}(G;x+1,y,z)=&x^{-k(G)}(yz)^{-v(G)}Z(G;xyz^{2},\mathbf{yz},z^{-1}).\label{eq:GenerRs}
\end{align}
As a consequence, it is a generalisation of the \BRp{} as $R_{s}$ reduces to the latter if all the edges of $G$ are positive.

Whereas the polynomial $Z$ appears naturally if one looks for a multivariate generalisation of the signed \BRp{}, it can also be expressed in terms of the \emph{unsigned} multivariate \BRp{}\footnote{We thank our anonymous referee for having pointed this out to us.} introduced in \cite{Moffatt2008ab}. Actually there is no real difference between signed and unsigned polynomials at the multivariate level.\\

\noindent
Recall that the multivariate \BRp{} is defined as follows \cite{Moffatt2008ab}:
\begin{align}
  \hat{Z}(G;q,\mathbf{\beta},c)\defi&\sum_{F\subseteq G}q^{k(F)}\Big(\prod_{e\in E(F)}\beta_{e}\Big)c^{f(F)}.\label{eq:MBRDef}
\end{align}
Considering, now, a signed ribbon graph $G$, we can take advantage of the natural partition of the set of edges into positive and negative ones to recover the signed \BRp{}. To this end, we have to choose particular weights in accordance with the partition. With the following choice,
\begin{align}
  \beta_{e}=&%
  \begin{cases}
    \alpha_{e}&\text{if $e$ is positive,}\\
    q\alpha_{e}^{-1}&\text{if $e$ is negative,}
  \end{cases}\label{eq:weightChoice}\\
  \intertext{the signed polynomial $Z$ is given by}
  Z(G;q,\balpha,c)=&\Big(\prod_{e\in E_{-}(G)}q^{-1/2}\alpha_{e}\Big)\hat{Z}(G;q,\mathbf{\beta},c).\label{eq:ZZhat}
\end{align}
The proof of \eqref{eq:ZZhat} relies on the following equalities: $E_{-}(\Fb)=E_{-}(G)\setminus E_{-}(F)$ and\\
$s(F)=e_{-}(F)-\frac 12 e_{-}(G)$.

Despite the equality \eqref{eq:ZZhat} we decide to use the polynomial $Z$ instead of $\hat{Z}$. The former arises indeed naturally from the signed \BRp{} which has an interesting behaviour with respect to the partial duality. Moreover, the sign dependence is more explicit in $Z$. It is true that some of the proofs we give in the rest of this article may be made shorter by using $\hat{Z}$ instead, but we think that it is interesting to demonstrate the role of the signed character of the polynomial.\\

The multivariate signed polynomial (\ref{eq:SMBRDef}) is also a signed generalisation of the multivariate dichromatic polynomial. Recall that this is defined by
\begin{align}
  Z_{T}(G;q,\balpha)\defi&\sum_{F\subseteq G}q^{k(F)}\prod_{e\in E(F)}\alpha_{e}.\label{eq:MultiTutte}
\end{align}
We will use this fact to prove that the duality relation for this multivariate Tutte polynomial is a consequence of the duality relation for $Z$.

In \cite{Chmutov2007ab}, S.~Chmutov and I.~Pak noted that $R_{s}$ is a generalisation of the signed Tutte polynomial defined by Kauffman in \cite{Kauffman1989aa}. As a consequence, $Z$ is also a generalisation of the Kauffman's polynomial $Q$. Indeed the latter can be expressed as an evaluation of $Z$: for any $e\in E(G)$, let $\alpha_{e}=Ad$ and let us write $\mathbf{Ad}$ for the corresponding set $\{\alpha_{e}\}_{e\in E(G)}$. Then we have:
\begin{align}
  Q[G](A,1,d)=&d^{-v(G)-1-k(G)}A^{k(G)}Z(G;d^{2},\mathbf{Ad},1).\label{eq:GenerKauff}
\end{align}

\subsection{Simple properties}
\label{sec:simple-properties}

\begin{prop}[Disjoint union, one-point join]
  Let $G_{1}\cup G_{2}$ be the disjoint union of $G_{1}$ and $G_{2}$. Then
  \begin{subequations}
    \begin{align}
      Z(G_{1}\cup G_{2};q,\mathbf{\alpha},c)=&Z(G_{1};q,\mathbf{\alpha_{1}},c)Z(G_{2};q,\mathbf{\alpha_{2}},c)\label{eq:DisjUnion}
    \end{align}
    where $\mathbf{\alpha}=\mathbf{\alpha_{1}}\cup\mathbf{\alpha_{2}}$.\\
    
    Let $G_{1}\cdot G_{2}$ be the one-point join of $G_{1}$ and $G_{2}$. Then
    \begin{align}
      Z(G_{1}\cdot G_{2};q,\mathbf{\alpha},c)=&\frac{1}{qc}Z(G_{1};q,\mathbf{\alpha_{1}},c)Z(G_{2};q,\mathbf{\alpha_{2}},c).\label{eq:1PtJoin}
    \end{align}
  \end{subequations}
\end{prop}
The proof follows essentially \cite{Bollobas2002aa}. $G$ being the disjoint union of $G_{1}$ and $G_{2}$, any of its spanning subgraphs $F$ is the disjoint union of a subgraph $F_{1}$ of $G_{1}$ and a subgraph $F_{2}$ of $G_{2}$. The parameters $k,s$ and $f$ are additive under the disjoint union and $E_{\pm}(F_{1}\cup F_{2})=E_{\pm}(F_{1})\cup E_{\pm}(F_{2})$.

If $G$ is the one-point join of $G_{1}$ and $G_{2}$ then for any of its subgraphs $F$, there exists subgraphs $F_{1}$ of $G_{1}$ and $F_{2}$ of $G_{2}$ such that $F=F_{1}\cdot F_{2}$. To prove (\ref{eq:1PtJoin}), we just have to remark that $k(F_{1}\cdot F_{2})=k(F_{1})+k(F_{2})-1$ and $f(F_{1}\cdot F_{2})=f(F_{1})+f(F_{2})-1$, the function $s$ being additive.

\begin{rem}
  If one defines $\widetilde{Z}(G;q,\mathbf{\alpha},c)\defi q^{-k(G)}c^{-f(G)}Z(G;q,\mathbf{\alpha},c)$ then
  \begin{align}
    \widetilde{Z}(G_{1}\cup G_{2};q,\mathbf{\alpha},c)=&\widetilde{Z}(G_{1}\cdot G_{2};q,\mathbf{\alpha},c)=\widetilde{Z}(G_{1};q,\mathbf{\alpha_{1}},c)\widetilde{Z}(G_{2};q,\mathbf{\alpha_{2}},c).
  \end{align}
\end{rem}

\subsection{Contractions and deletions}
\label{sec:contr-delet}

\begin{prop}[Deletion and contraction]\label{prop:DelCont}
  Let $G$ be any signed ribbon graph and for any edge $e\in E(G)$, let $\balpha_{e}\defi\balpha\setminus\{\alpha_{e}\}$. Then for every positive edge $e$ of $G$ which is not an orientable loop,
  \begin{subequations}
    \begin{align}
      Z(G;q,\balpha,c)=&\alpha_{e}Z(G/e;q,\balpha_{e},c)+Z(G- e;q,\balpha_{e},c).\label{eq:DelContOrdBridNOP}
      \intertext{For every positive orientable \emph{trivial} loop $e$,}
      Z(G;q,\balpha,c)=&q^{-1}\alpha_{e}Z(G/e;q,\balpha_{e},c)+Z(G- e;q,\balpha_{e},c)\label{eq:DelContOTrivP}\\
      =&(\alpha_{e}c+1)Z(G-e;q,\balpha_{e},c).\label{eq:DelOTrivP}
      \intertext{For every negative edge $e$ of $G$ which is not an orientable loop,}
      Z(G;q,\balpha,c)=&q^{1/2}Z(G/e;q,\balpha_{e},c)+q^{-1/2}\alpha_{e}Z(G-e;q,\balpha_{e},c).\label{eq:DelContOrdBridNON}
      \intertext{For every negative orientable \emph{trivial} loop $e$,}
      Z(G;q,\balpha,c)=&q^{-1/2}\big(Z(G/e;q,\balpha_{e},z)+\alpha_{e}Z(G-e;q,\balpha_{e},z)\big)\label{eq:DelContOTrivN}\\
      =&(q^{1/2}c+q^{-1/2}\alpha_{e})Z(G-e;q,\balpha_{e},c).\label{eq:DelOTrivN}
    \end{align}
  \end{subequations}
\end{prop}
\begin{proof}
  Let $e\in E(G)$ be either an ordinary edge, a bridge or a \emph{non-orientable} loop. We have
  \begin{align}
    Z(G;q,\mathbf{\alpha},c)=&\sum_{F\subseteq G}q^{k(F)+s(F)}\Big(\prod_{e'\in E_{+}(F)}\alpha_{e'}\prod_{e'\in E_{-}(\bar{F})}\alpha_{e'}\Big)c^{f(F)}\fide\sum_{F\subseteq G}M(F,\balpha)\\
    =&\sum_{\substack{F\subseteq G\tq\\e\in E(F)}}M(F,\balpha)+\sum_{\substack{F\subseteq G\tq\\e\notin E(F)}}M(F,\balpha)
  \end{align}
  The subgraphs of $G$ which contain (resp.\@ do not contain) $e$ are in one-to-one correspondence with the subgraphs of $G/e$ (resp.\@ $G-e$).
  Let $F\subseteq G$ such that $e\in E(F)$. We have: $k(F)=k(F/e)$ and $f(F)=f(F/e)$. The table \ref{tab:SignDepProp} lists some sign-dependent equalities concerning $s$ and the $\alpha$'s. Note that they are true for \emph{any} type of edge.
 \begin{table}[htb]
   \centering
    \begin{tabular}{|c|l|l|}
        \hhline{~--}
        \multicolumn{1}{c|}{}&\multicolumn{1}{|c|}{\rule[-7pt]{0pt}{20pt}$\displaystyle e\in E(F)$}&\multicolumn{1}{|c|}{\rule[-3pt]{0pt}{15pt}$\displaystyle e\notin E(F)$}\\
        \hline
        \multirow{2}{20mm}{$\displaystyle\veps(e)=1$}&\rule[-3pt]{0pt}{15pt}$\displaystyle\bullet\; s(F)=s(F/e)$&$\bullet\; s(F)=s(F-e)$\\
        &$\bullet\; E_{+}(F)=E_{+}(F/e)\cup\{e\}$&\rule[-7pt]{0pt}{20pt}\\
        \hhline{|-|-|-|}
        \multirow{2}{20mm}{$\displaystyle\veps(e)=-1$}&\rule[-7pt]{0pt}{20pt}$\bullet\; s(F)=s(F/e)+1/2$&$\bullet\; s(F)=s(F/e)-1/2$\\
        &&$\bullet\; E_{-}(\Fb)=E_{-}(\Fb/e)\cup\{e\}$\\
        \hline
      \end{tabular}
      \caption{Sign-dependent properties}
      \label{tab:SignDepProp}
    \end{table}\\

    We then have
    \begin{align}
      Z(G;q,\mathbf{\alpha},c)=&
      \begin{cases}
       \displaystyle\alpha_{e}\sum_{\substack{F\subseteq G/e}}M(F,\balpha_{e})+\sum_{\substack{F\subseteq G-e}}M(F,\balpha_{e})&\text{if $e$ is positive,}\\
       \displaystyle q^{1/2}\sum_{\substack{F\subseteq G/e}}M(F,\balpha_{e})+q^{-1/2}\alpha_{e}\sum_{\substack{F\subseteq G-e}}M(F,\balpha_{e})&\text{if $e$ is negative}
      \end{cases}
    \end{align}
    which proves \eqref{eq:DelContOrdBridNOP} and \eqref{eq:DelContOrdBridNON}.\\

    Let us now consider an orientable trivial loop $e$. Let $F$ be a subgraph of $G$ containing $e$. We have $k(F)=k(F/e)-1$, $f(F)=f(F/e)$, $k(F)=k(F-e)$ and $f(F)=f(F-e)+1$. Together with the table \ref{tab:SignDepProp}, this proves the equations \eqref{eq:DelContOTrivP}, \eqref{eq:DelOTrivP}, \eqref{eq:DelContOTrivN} and \eqref{eq:DelOTrivN}.
  \end{proof}

The preceding proposition applies to all types of edges except the orientable nontrivial loops. For such edges, there is no simple formula like those of proposition \ref{prop:DelCont}. Indeed let $e$ be an orientable nontrivial loop, and let $F$ be a subgraph of $G$ such that $e\in E(F)$. The relationship between $k(F)$ and $k(F/e)$ (or $k(F-e)$) is $F$-dependent. The same holds for the number of faces $f$\footnote{I thank S.~Chmutov for having explained to me this point.}.\\

In some cases, the equations \eqref{eq:DelContOrdBridNOP} and \eqref{eq:DelContOrdBridNON} can be further simplified:
\begin{cor}\label{cor:DelCont2}
  Let $G$ be any ribbon graph. Then for every positive bridge $e$,
  \begin{subequations}
    \begin{align}
      Z(G;q,\balpha,c)=&(\alpha_{e}+qc)Z(G/e;q,\balpha_{e},c).\label{eq:BridgeP}
      \intertext{For every positive non-orientable trivial loop $e$}
      Z(G;q,\balpha,c)=&(\alpha_{e}+1)Z(G-e;q,\balpha_{e},c).\label{eq:NOTrivP}
      \intertext{For every negative bridge $e$}
      Z(G;q,\balpha,c)=&q^{1/2}(1+\alpha_{e}c)Z(G/e;q,\balpha_{e},c).\label{eq:BridgeN}
      \intertext{For every negative non-orientable trivial loop $e$}
      Z(G;q,\balpha,c)=&(q^{1/2}+q^{-1/2}\alpha_{e})Z(G-e;q,\balpha_{e},c).\label{eq:NOTrivN}
    \end{align}
  \end{subequations}
\end{cor}
\begin{proof}
  For bridges, the argument is the usual one (see \cite{Bollobas2002aa} for example). If $e$ is a bridge, then $G-e$ is the disjoint union of two ribbon graphs $G_{1}$ and $G_{2}$. Then, using the equations \eqref{eq:DisjUnion} and \eqref{eq:1PtJoin}, we prove \eqref{eq:BridgeP} and \eqref{eq:BridgeN}.

If $e$ is a non-orientable trivial loop, then $G/e$ and $G-e$ are two different one-point joins of the same two graphs \cite{Chmutov2007aa}. As a consequence, their (multivariate signed) \BRp s are equal to each other.
\end{proof}

\begin{prop}[Shift of the weights]\label{prop:ShiftWeights}
  Let $E(G)\fide\{e_{i}\}_{i=1,\dots,e(G)}$ be the set of edges of $G$. Let $\mathbf{\alpha}+1_{i}=\{\alpha_{1},\dots,\alpha_{i}+1,\dots,\alpha_{e(G)}\}$ be the weights of $E(G)$ where $\alpha_{i}$ has been shifted by one. Then
  \begin{align}
    \intertext{if $e_{i}$ is positive and not an orientable loop,}
    Z(G;q,\mathbf{\alpha}+1_{i},c)=&Z(G;q,\mathbf{\alpha},c)+Z(G/e_{i};q,\mathbf{\alpha}_{i},c),\label{eq:ShiftWeight+}
    \intertext{if $e_{i}$ is a positive orientable trivial loop,}
    Z(G;q,\mathbf{\alpha}+1_{i},c)=&Z(G;q,\mathbf{\alpha},c)+q^{-1}Z(G/e_{i};q,\mathbf{\alpha}_{i},c),\label{eq:ShiftWeightOT}
    \intertext{and for any negative edge $e_{i}$,}
    Z(G;q,\mathbf{\alpha}+1_{i},c)=&Z(G;q,\mathbf{\alpha},c)+q^{-1/2}Z(G- e_{i};q,\mathbf{\alpha}_{i},c)\label{eq:ShiftWeight-}
  \end{align}
  where $\mathbf{\alpha}_{i}=\{\alpha_{1},\dots,\alpha_{i-1},\alpha_{i+1},\dots,\alpha_{e(G)}\}$.
\end{prop}
\begin{proof}
  Let $e_{i}\in E(G)$. We have
  \begin{align}
    Z(G;q,\mathbf{\alpha}+1_{i},c)=&\left.Z(G;q,\mathbf{\alpha}+\zeta 1_{i},c)\rabs_{\zeta=1}\nonumber\\
      =&Z(G;q,\mathbf{\alpha},c)+\int_{0}^{1}d\zeta\,\frac{dZ}{d\zeta}(G;q,\mathbf{\alpha}+\zeta 1_{i},c). 
  \end{align}
  We now focus on the derivative term. We distinguish between three different cases:
  \begin{enumerate}
  \item $\veps(e_{i})=1$ and $e_{i}$ is not an orientable loop: the only non-vanishing term under derivation in the sum (\ref{eq:SMBRDef}) corresponds to the subgraphs $F$ such that $e_{i}\in E(F)$. The sum is then in one-to-one correspondence with the sum over the subgraphs of $G/e_{i}$. We have
    \begin{align}
      \frac{dZ}{d\zeta}(G;q,\mathbf{\alpha}+\zeta 1_{i},c)=&\sum_{F\subseteq G\tq e_{i}\in E(F)}q^{k(F)+s(F)}\Big(\prod_{e\in E_{+}(F)-\{ e_{i}\}}\alpha_{e}\prod_{e\in E_{-}(\bar{F})}\alpha_{e}\Big)c^{f(F)}.
    \end{align}
    From $F$ to $F/e_{i}$, $k,f$ and $s$ do not change. The integration over $\zeta$ equals one which proves (\ref{eq:ShiftWeight+}).
  \item $e_{i}$ is a positive orientable trivial loop: the only difference with the previous case is that $k(F)=k(F/e_{i})-1$ which proves \eqref{eq:ShiftWeightOT}.
  \item $\veps(e_{i})=-1$: the non-vanishing terms correspond to the subgraphs $F$ which do not contain $e_{i}$. The edge $e_{i}$ belongs, then, to $\bar{F}$. The sum is in one-to-one correspondence with the sum over the subgraphs of $G- e_{i}$. $e_{i}$ being negative, we have $s(F)=s(F- e_{i})-\frac 12$ and we get (\ref{eq:ShiftWeight-}).
  \end{enumerate}
\end{proof}
\begin{rem}
  Let $Z'(G;q,\mathbf{\alpha},c)\defi q^{s(G)}Z(G;q,\mathbf{\alpha},c)$. Then the same proposition holds but with a factor $1$ instead of $q^{-1/2}$ in (\ref{eq:ShiftWeight-}).
\end{rem}


\section{Tree expansions}
\label{sec:spann-tree-expans}

\subsection{Spanning tree expansion}
\label{sec:spann-tree-expans-1}

The original \BRp{} can be defined by a spanning tree expansion (see \cite{Bollobas2002aa}, section $6$).

Given a graph $G$, a spanning tree is a connected spanning subgraph with vanishing nullity. For the sake of completeness we recall the definitions of  the activities involved in this spanning tree expansion.
\begin{defn}[Activities wrt a spanning tree]
  Let $G$ be a connected ribbon graph and $\prec$ be an order on the set $E(G)$ of edges of $G$. Let $T$ be a spanning tree of $G$. Let $e\in E(T)$, we write $U_{T}(e)$ for the \emph{cut} defined by $e$:
  \begin{align}
    U_{T}(e)\defi&\lb f\in E(G)\setminus E(T)\tqs (T-e)+f \text{ is a spanning tree} \rb.
  \end{align}
  For $e\in E(G)\setminus E(T)$ we write $Z_{T}(e)$ for the cycle defined by $e$, namely the unique cycle of $T+e$.

  An edge $e\in E(T)$ is said \textbf{internally active} if it is the smallest edge (wrt $\prec$) in $U_{T}(e)$. Otherwise it is \textbf{internally inactive}. The number of internally active edges (wrt $T$ and $\prec$) is denoted by $\mathbf{i(T)}$.

  An edge $e\in E(G)\setminus E(T)$ is said \textbf{externally active} if it is the smallest edge (wrt $\prec$) in $Z_{T}(e)$. Otherwise it is \textbf{externally inactive}. The set of externally active edges of $G$ is $\mathbf{EA(T)}$ with $\mathbf{j(T)}\defi\card EA(T)$.
\end{defn}

Given an order on the set of edges of a ribbon graph $G$,
\begin{align}
  R(G;x+1,y,z,w)=&\sum_{F\subseteq G}x^{k(F)-k(G)}y^{n(F)}z^{k(F)-f(F)+n(F)}w^{t(F)}\label{eq:BRDef}\\
  =&\sum_{T}(x+1)^{i(T)}\sum_{S\subset EA(T)}y^{n(T\cup S)}z^{1-f(T\cup S)+n(T\cup S)}w^{t(T\cup S)}.\label{eq:BRSpanExp}
\end{align}
Clearly, for $z=w=1$, the sum over $S$ reduces to $(y+1)^{j(T)}$ (thanks to $n(T\cup S)=e(S)$). This is then the spanning tree expansion of the Tutte polynomial. But contrary to this one, the spanning tree expansion of the \BRp{} cannot be expressed as easily. This is partly due to the lack of reduction relations for the nontrivial loops.\\

In this section we give a \ste{}, similar to (\ref{eq:BRSpanExp}), for the multivariate signed polynomial. We restrict ourselves to connected graphs but the extension to all ribbon graphs is trivial. So let $G$ be any connected ribbon graph and $T$ a spanning tree of $G$. The set $E(G)$ is endowed with an order $\prec$. We define the following subsets of $E(G)$:
\begin{itemize}
\item the subset of positive (resp.\@ negative) internally active edges: $IA_{\pm}(T)$,
\item the subset of positive (resp.\@ negative) internally inactive edges: $II_{\pm}(T)$,
\item the subset of positive (resp.\@ negative) externally active edges: $EA_{\pm}(T)$,\\
  $EA(T)\defi EA_{+}(T)\cup EA_{-}(T)$,
\item the subset of positive (resp.\@ negative) externally inactive edges: $EI_{\pm}(T)$,
\item the subset of positive (resp.\@ negative) trivial orientable loops: $TO_{\pm}(G)$,\\
  $TO\defi TO_{+}(G)\cup TO_{-}(G)$ and
\item the subset of positive (resp.\@ negative) trivial non-orientable loops: $TNO_{\pm}(G)$,\\
  $TNO\defi TNO_{+}(G)\cup TNO_{-}(G)$.
\end{itemize}
\vspace{\baselineskip}%
Let $w(G,\prec;q,\balpha,c)$ be the following polynomial:
\begin{align}
  w(G,\prec)\defi& q^{k(G)}\Big(\prod_{e\in TO_{+}}(\alpha_{e}c+1)\Big)\Big(\prod_{e\in TO_{-}}\sqrt q(c+\alpha_{e}/q)\Big)\nonumber\\
  &\Big(\prod_{e\in TNO_{+}}(\alpha_{e}+1)\Big)\Big(\prod_{e\in TNO_{-}}\sqrt q(\alpha_{e}/q+1)\Big)\nonumber\\
  &\sum_{T}\Big(\prod_{e\in IA_{+}(T)}(\alpha_{e}+qc)\Big)\Big(\prod_{e\in IA_{-}(T)}\sqrt{q}(1+\alpha_{e}c)\Big)
  \Big(\prod_{e\in II_{+}(T)}\alpha_{e}\Big)\sqrt{q}^{\labs II_{-}(T)\rabs}\nonumber\\
  &\Big(\prod_{e\in EI_{-}(T)}\alpha_{e}/\sqrt q\Big)\sum_{\substack{S\subset\\EA(T)-TO-TNO}}q^{s(S)}\Big(\prod_{\substack{e\in E_{+}(S)\cup\\E_{-}(\bar{S})}}\alpha_{e}\Big)c^{f(T\cup S)}\label{eq:SpanExpMultiBR}
\end{align}
where the first sum runs over all spanning trees in $G$ and $\bar{S}$ is the complement of $S$ in $EA(T)-TO-TNO$.
\begin{thm}\label{thm:SpanExpMultiBR}
  For any connected ribbon graph $G$ and any order $\prec$ on $E(G)$,\\
  $w(G,\prec;q,\balpha,c)=Z(G;q,\balpha,c)$.
\end{thm}
\begin{proof}
  It is very similar to the one of equation (\ref{eq:BRSpanExp}) in \cite{Bollobas2002aa}. The proof is made by induction on the number of edges of $G$ which are not loops. If $G$ has no such edges, it is a one-vertex ribbon graph and all its edges are externally active. In this case, $k(G)=k(S)$ for all subsets $S$ and there is only one spanning tree namely $(V(G),\emptyset)$. Then the expression (\ref{eq:SpanExpMultiBR}) for $w(G,\prec)$ equals the definition of the multivariate signed polynomial (\ref{eq:SMBRDef}) after the use of proposition \ref{prop:DelCont} and corollary \ref{cor:DelCont2}.\\

Otherwise, if $G$ has edges which are not loops, we choose the last edge $e$ in the order $\prec$. If $e$ is a bridge, every spanning tree contains $e$ and it is always internally active. The sum over $T(G)$ is in one-to-one correspondence with the sum over the spanning trees of $G/e$. The contraction of $e$ does not affect the activities of the other edges:
\begin{align}
  w(G,\prec)=&%
  \begin{cases}
    (\alpha_{e}+qc)\,w(G/e,\prec)&\text{if $e$ is positive,}\\
    \sqrt q(1+\alpha_{e}c)\,w(G/e,\prec)&\text{if $e$ is negative.}
  \end{cases}
  \label{eq:SpTrBridge}
\end{align}
If $e$ is ordinary, it is neither internally nor externally active. Its contraction or deletion does not change the activities of the other edges. When $e$ belongs to a spanning tree of $G$, it is internally inactive whereas when it does not belong to a tree, it is externally inactive. Thus we have:
\begin{align}
  w(G,\prec)=&%
  \begin{cases}
    \alpha_{e}w(G/e,\prec)+w(G-e,\prec)&\text{if $e$ is positive,}\\
    \sqrt q w(G/e,\prec)+\frac{\alpha_{e}}{\sqrt q}w(G-e,\prec)&\text{if $e$ is negative.}
  \end{cases}
  \label{eq:SpTrOrd}
\end{align}
As a consequence, $w(G,\prec)$ equals the polynomial (\ref{eq:SMBRDef}) when $G$ has only loops. When $G$ has not only loops, $w(G,\prec)$ obeys the same reduction relations as the multivariate signed \BRp{}. These relations allow one to express $w$ as a (weighted) sum of contributions of one-point graphs. This proves the theorem.
\end{proof}

\subsection{Quasi-tree expansion}
\label{sec:quasi-tree-expansion}

More recently, A.~Champanerkar, I.~Kofman and N.~Stoltzfus found another tree expansion for the \BRp{} \cite{Champanerkar2007aa}. Its advantage over the spanning tree expansion of B.~Bollob\'as and O.~Riordan \cite{Bollobas2002aa} is that it requires fewer summands and the associated weights are defined topologically. We now recall this new \emph{quasi-tree} expansion. Then we will give its multivariate analogue.\\

The quasi-tree expansion in \cite{Champanerkar2007aa}, being only valid for orientable ribbon graphs, we restrict ourselves to such a class in this subsection. Note that an orientable ribbon graph can always be drawn with untwisting edges.
\begin{defn}[Quasi-tree]
  Let $G$ be an orientable ribbon graph. A quasi-tree $Q$ is a spanning subgraph of $G$ with $f(Q)=1$. The set of quasi-trees in $G$ is denoted by $\mathbf{\cQ_{G}}$.
\end{defn}
Any orientable ribbon graph $G$ can be represented by a cyclic graph namely a set of half-edges $\bbH$, a fixed-point free involution $\sigma_{1}$ and a permutation $\sigma_{0}$ of $\bbH$. The cycles of $\sigma_{0}$ form the vertex set of $G$, those of $\sigma_{1}$ its edges. The faces of $G$ are given by the orbits of $\sigma_{2}\defi \sigma_{1}\circ\sigma_{0}^{-1}$.

Given a total order on the edges of $G$, one can define the activities wrt a quasi-tree. To this end, the authors of \cite{Champanerkar2007aa} proved the following proposition:
\begin{prop}
  Let $G$ be a connected orientable ribbon graph. Every quasi-tree $Q$ of $G$ corresponds to the ordered chord diagram $C_{Q}$ with consecutive markings in the positive direction given by the following permutation on $\bbH$:
  \begin{align}
    \sigma(i)\defi&%
    \begin{cases}
      \sigma_{0}(i)&\text{if $i\notin Q$,}\\
      \sigma_{2}^{-1}(i)&\text{if $i\in Q$}.
    \end{cases}
  \end{align}
\end{prop}
\begin{defn}[Activities wrt a quasi-tree]
  Given a connected orientable ribbon graph $G$ and a quasi-tree $Q$ of $G$, an edge of $G$ is \textbf{internal} if it belongs to $E(Q)$ and \textbf{external} otherwise. Moreover an edge is said \textbf{live} if its corresponding chord in $C_{Q}$ does not intersect any lower-ordered chord. If it does, the edge is called \textbf{dead}.

One lets $\mathbf{\cD(Q)}$ denote the spanning subgraph of $G$, the edges of which are the internally dead edges. $\mathbf{\cI(Q)}$ is the set of internally live edges and $\mathbf{\cE(Q)}$ the set of externally live edges.
\end{defn}
For a given quasi-tree $Q$, one defines the graph (not the ribbon graph) $\mathbf{G_{Q}}$ as the graph, the vertices of which are the components of $\cD(Q)$, and the edges of which are the internally live edges of $G$. One can now state the main theorem of \cite{Champanerkar2007aa}:
\begin{thmb}[Quasi-tree expansion of the \BRp{} \cite{Champanerkar2007aa}]
  Let $G$ be a connected orientable ribbon graph. With the preceding definitions, we have:
  \begin{align}
    R(G;x,y,z)=&\sum_{Q\subset G}y^{n(\cD(Q))}z^{2g(\cD(Q))}(1+y)^{|\cE(Q)|}\,T(G_{Q};x,1+yz^{2})
  \end{align}
  where $\displaystyle T(G_{Q},x,y)=\sum_{F\subset G_{Q}}(x-1)^{r(G_{Q})-r(F)}(y-1)^{n(F)}$ is the Tutte polynomial of $G_{Q}$.
\end{thmb}
In order to prove this theorem, the authors of \cite{Champanerkar2007aa} proved a series of results. We gather, in the following lemma, those results which we need in order to prove the quasi-tree expansion of $Z$:
\begin{lemma}\label{lem:QTSteps}
  Let $G$ be a connected orientable ribbon graph and $\cS_{G}$ its set of spanning subgraphs. Then $\cS_{G}$ is in one-to-one correspondence with $\bigcup_{Q\in\cQ_{G}}\cI(Q)\times\cE(Q)$. Namely, to each spanning subgraph $F$ there corresponds a unique quasi-tree $Q_{F}$. Then $E(F)=\cD(Q_{F})\cup S,\,S\subset \cI(Q_{F})\cup\cE(Q_{F})$. In addition, for a given quasi-tree $Q$, let $S=S_{1}\cup S_{2},\,S_{1}\subset\cI(Q)$ and $S_{2}\subset\cE(Q)$. With a slight abuse of notation, we have:
  \begin{itemize}
  \item $k(\cD\cup S)=k(\cD\cup S_{1})=k(W)$, where $W$ is the spanning subgraph of $G_{Q}$ the edge-set of which is $S_{1}$,
  \item $f(\cD\cup S)=f(\cD)-|S_{1}|+|S_{2}|$.
  \end{itemize}
\end{lemma}
We now state and prove the quasi-tree expansion of the multivariate signed polynomial $Z$:
\begin{prop}
  Let $G$ be a signed connected orientable ribbon graph. With weights $\mathbf{\beta}$ given by \eqref{eq:weightChoice}, we have:
  \begin{align}
    Z(G;q,\balpha,c)=&\Big(\prod_{e\in E_{-}(G)}q^{-1/2}\alpha_{e}\Big)\hat{Z}(G;q,\mathbf{\beta},c),\nonumber\\
    \hat{Z}(G;q,\mathbf{\beta},c)=&\sum_{Q\in\cQ_{G}}\Big(\prod_{e\in\cD(Q)}\beta_{e}\Big)c^{f(\cD(Q))}\Big(\prod_{e\in\cE(Q)}(1+c\beta_{e})\Big)\,Z_{T}(G_{Q};q,\mathbf{\beta}/c)\label{eq:QTZ}
  \end{align}
  where $Z_{T}$ is the multivariate Tutte polynomial defined in equation \eqref{eq:MultiTutte}.
\end{prop}
\begin{proof}
  \begin{align}
    \hat{Z}(G;q,\mathbf{\beta},c)=&\sum_{F\subset G}q^{k(F)}\Big(\prod_{e\in E(F)}\beta_{e}\Big)c^{f(F)}\\
    =&\sum_{Q\in\cQ_{G}}\sum_{S_{1}\subset\cI(Q)}\sum_{S_{2}\subset\cE(Q)}q^{k(\cD(Q)\cup S_{1}\cup S_{2})}\Big(\prod_{e\in \cD(Q)\cup S_{1}\cup S_{2}}\beta_{e}\Big)c^{f(\cD(Q)\cup S_{1}\cup S_{2})}\\
    =&\sum_{Q\in\cQ_{G}}\Big(\prod_{e\in\cD(Q)}\beta_{e}\Big)c^{f(\cD(Q))}\sum_{S_{2}\subset\cE(Q)}\Big(\prod_{e\in S_{2}}c\beta_{e}\Big) \sum_{S_{1}\subset\cI(Q)}q^{k(\cD(Q)\cup S_{1})}\Big(\prod_{e\in S_{1}}\beta_{e}/c\Big)\\
    =&\sum_{Q\in\cQ_{G}}\Big(\prod_{e\in\cD(Q)}\beta_{e}\Big)c^{f(\cD(Q))}\Big(\prod_{e\in\cE(Q)}1+c\beta_{e}\Big) \sum_{W\subset G_{Q}}q^{k(W)}\Big(\prod_{e\in E(W)}\beta_{e}/c\Big)
  \end{align}
  which proves the proposition.
\end{proof}

\section{Partial duality}
\label{sec:generalised-duality-1}

We now state and prove our main theorem namely the invariance of the multivariate signed \BRp{} under partial duality.
\begin{thm}
  \label{thm:PartialDuality}
  Let $G$ be a ribbon graph. For any subset $E'\subset E(G)$, the multivariate signed \BRp{} (\ref{eq:SMBRDef}) at $q=1$ is invariant under the partial duality with respect to $E'$:
  \begin{align}
    \label{eq:PartialDuality}
    Z(G;1,\balpha,c)=&Z(G^{E'};1,\balpha,c).
  \end{align}
\end{thm}
\begin{rem}
  The duality transformation of the signed \BRp{} \cite{Chmutov2007aa} is a consequence of our multivariate version.
\end{rem}
\begin{proof}
  We follow the steps of the proof given by S.~Chmutov in \cite{Chmutov2007aa} for the signed \BRp{} (\ref{eq:BRPolyDef}). Let us recall that
  \begin{align}
    Z(G;q,\mathbf{\alpha},c)=&\sum_{F\subseteq G}q^{k(F)+s(F)}\Big(\prod_{\substack{e\in E_{+}(F)\\\cup E_{-}(\bar{F})}}\alpha_{e}\Big)c^{f(F)}\fide\sum_{F\subseteq G}M(F).
  \end{align}
  To any spanning subgraph $F\subseteq G$, we associate a spanning subgraph $F'\subseteq G^{E'}$ the edge-set of which is $E(F')\defi (E'\cup E(F))-(E'\cap E(F))$. This correspondence is one-to-one so that it is enough to prove that $\left.M(F)\rabs_{q=1}=\left.M(F')\rabs_{q=1}$. Moreover, thanks to lemma \ref{lem:SimpleProp}, it is sufficient to consider the case when $E'$ is reduced to a single edge $e'$. We can also assume that $e'\in E(F)$ (so that $e'\notin E(F')$) because if not, the roles of $G$ and $G^{\{e'\}}$ are simply interchanged.

We now compare the parameters $k,s$ and $f$ for the subgraphs $F$ and $F'$. By construction, $f(F)=f(F')$. Let us first assume that $e'$ is positive in $F$. By assumption, $e'\in E(F)$ and $F'=(V(G^{\{e'\}}),E(F)-\{e'\})$. Then $s(F)=s(F')+1/2$, $E_{+}(F)=E_{+}(F')\cup\{e'\}$ and $E_{-}(\Fb)=E_{-}(\bar{F'})-\{e'\}$. In the case of $e'$ being negative, $s(F)=s(F')+1/2$, $E_{+}(F)=E_{+}(F')$ and $E_{-}(\Fb)=E_{-}(\bar{F'})$. We then have
\begin{align}
  q^{k(F)+s(F)}\Big(\prod_{\substack{e\in E_{+}(F)\\\cup E_{-}(\bar{F})}}\alpha_{e}\Big)c^{f(F)}=&q^{k(F')+s(F')+1/2}\Big(\prod_{\substack{e\in E_{+}(F')\\\cup E_{-}(\bar{F'})}}\alpha_{e}\Big)c^{f(F')}
\end{align}
which proves the theorem.
\end{proof}

Thanks to theorem \ref{thm:PartialDuality}, we can prove a \emph{weak} contraction/deletion reduction relation (meaning only true for $q=1$) for the orientable nontrivial edges.
\begin{lemma}
  \label{lem:CDNONTEdge}
  Let $G$ be a ribbon graph and $e\in E(G)$ a nontrivial orientable loop. Then
  \begin{align}
    Z(G;1,\balpha,c)=&
    \begin{cases}
      \alpha_{e}Z(G/e;1,\balpha_{e},c)+Z(G-e;1,\balpha_{e},c)&\text{if e is positive,}\\
      Z(G/e;1,\balpha_{e},c)+\alpha_{e}Z(G-e;1,\balpha_{e},c)&\text{if e is negative.}
    \end{cases}
  \end{align}
\end{lemma}
\begin{proof}
  From the theorem, $Z(G;1,\alpha,c)=Z(G^{\{e\}};1,\alpha,c)$. The edge $e$ being nontrivial in $G$, it is ordinary in $G^{\{e\}}$. We can then apply the proposition \ref{prop:DelCont} to $Z(G^{\{e\}})$. If $e$ is positive in $G$, it is negative in $G^{\{e\}}$ and we use equation (\ref{eq:DelContOrdBridNON}):
  \begin{align}
    Z(G;1,\balpha,c)=&Z(G^{\{e\}};1,\balpha,c)=Z(G^{\{e\}}/e;1,\balpha_{e},c)+\alpha_{e}Z(G^{\{e\}}-e;1,\balpha_{e},c)\\
    =&Z(G-e;1,\balpha_{e},c)+\alpha_{e}Z(G/e;1,\balpha_{e},c).
  \end{align}
  If $e$ is negative in $G$, we use equation (\ref{eq:DelContOrdBridNOP}) instead:
  \begin{align}
    Z(G;1,\balpha,c)=&Z(G^{\{e\}};1,\balpha,c)=\alpha_{e}Z(G^{\{e\}}/e;1,\balpha_{e},c)+Z(G^{\{e\}}-e;1,\balpha_{e},c)\\
    =&\alpha_{e}Z(G-e;1,\balpha_{e},c)+Z(G/e;1,\balpha_{e},c).
  \end{align}
\end{proof}

\section{Natural duality}
\label{sec:natural-duality}

Let $G$ be an \emph{unsigned} ribbon graph. The usual dual $G^{\star}$ is then equivalent to $G^{E(G)}$. Let $R(G;x+1,y,z,w)$ be the \BRp{} (\ref{eq:BRDef}) for unsigned ribbon graphs. In \cite{Ellis-Monaghan2009aa,Moffatt2008ab}, a duality relation has been proved for $R$, namely
\begin{align}
  x^{g(G)}R(G;x+1,y,1/\sqrt{xy},1)=&y^{g(G)}R(G^{\star};y+1,x,1/\sqrt{xy},1).\label{eq:BRDuality}
\end{align}
This duality takes place on the surface $xyz^{2}=1$ which is the equivalent of our $q=1$. It is then a natural question as to whether the partial duality can reproduce this result. This has been addressed in \cite{Chmutov2007aa}. Taking into account the fact that the signed polynomial (\ref{eq:BRPolyDef}) reduces to the unsigned \BRp{} for graphs with only positive edges, we can use the (not so) partial duality with $E'=E(G)$. But remember that during this duality process, all the signs are changed. This mean that, starting with positive edges, our dual $G^{E}$ has only negative edges. So, to go to the \BRp{} for $G^{\star}$ we have to flip all the signs once more. Fortunately, one can prove a simple formula for that. The \emph{natural duality} from $G$ to $G^{\star}$ is then defined as the two following steps: a duality with respect to $E(G)$ and a change of the sign function $\veps_{G^{\star}}\defi -\veps_{G^{E}}$.\\

Here we study the behaviour of our multivariate polynomial $Z$ under the natural duality.

\subsection{Natural duality for the \BRp}
\label{sec:natural-duality-brp}

\begin{prop}[Flip of a sign]\label{prop:FlipSign}
  Let $G$ be a ribbon graph with sign function $\veps$, let $e_{i}\in E(G)$ and let $G_{-e_{i}}$ be the same ribbon graph but with a sign function $\veps_{-e_{i}}$ given by: $\forall e\in E(G_{-e_{i}})-\{e_{i}\},\,\veps_{-e_{i}}(e)=\veps(e)$ and $\veps_{-e_{i}}(e_{i})=-\veps(e_{i})$. Then
  \begin{align}
    Z(G_{-e_{i}};q,\mathbf{\alpha},c)=&\frac{\alpha_{e_{i}}}{\sqrt q}\,Z(G;q,\balpha^{i},c)
  \end{align}
  with $\balpha^{i}\defi\{\alpha_{1},\dots,\alpha_{i-1},\frac{q}{\alpha_{e_{i}}},\alpha_{i+1},\dots\}$.
\end{prop}
\begin{proof}
  Let $e_{i}\in E(G)$ such that $\veps_{-e_{i}}(e_{i})=1$. We have
  \begin{align}
    Z(G_{-e_{i}};q,\balpha,c)=&\sum_{F_{-e_{i}}\subseteq G_{-e_{i}}}q^{k(F_{-e_{i}})+s(F_{-e_{i}})}\Big(\prod_{e\in E_{+}(F_{-e_{i}})}\alpha_{e}\prod_{e\in E_{-}(\bar{F}_{-e_{i}})}\alpha_{e}\Big)c^{f(F_{-e_{i}})}\\
    \fide&\sum_{F_{-e_{i}}\subseteq G_{-e_{i}}} M(F_{-e_{i}},\balpha).
  \end{align}
  As usual, the sum is now divided into two parts corresponding, respectively, to the subgraphs which contain $e_{i}$ and to those which do not. So let $F_{-e_{i}}$ such that $e_{i}\in E(F_{-e_{i}})$. Then we have
  \begin{align}
    e_{-}(F_{-e_{i}})=&e_{-}(F)-1,\,E_{+}(F_{-e_{i}})=E_{+}(F)\cup\{e_{i}\},\, E_{-}(\bar{F}_{-e_{i}})=E_{-}(\bar{F})\text{ and}\nonumber\\
    Z(G_{-e_{i}};q,\balpha,c)=&\frac{\alpha_{e_{i}}}{\sqrt q}\sum_{F\tq e_{i}\in E(F)}M(F,\balpha)+\sum_{F_{-e_{i}}\tq e_{i}\notin E_{F_{-e_{i}}}}M(F_{-e_{i}},\balpha).
  \end{align}
  Let now $F_{-e_{i}}$ such that $e_{i}\in E(\bar{F}_{-e_{i}})$. In this case,
  \begin{align}
    e_{-}(\bar{F}_{-e_{i}})=&e_{-}(\bar{F})-1,\,E_{+}(F_{-e_{i}})=E_{+}(F),\,E_{-}(\bar{F}_{-e_{i}})=E_{-}(\bar{F})-\{e_{i}\}\text{ and}\nonumber\\
    Z(G_{-e_{i}};q,\balpha,c)=&\frac{\alpha_{e_{i}}}{\sqrt q}\sum_{F\tq e_{i}\in E(F)}M(F,\balpha)+\frac{\alpha_{e_{i}}}{\sqrt q}\sum_{F\tq e_{i}\notin E(F)}M(F,\balpha^{i}).
  \end{align}
In the first sum, the variable $\alpha_{e_{i}}$ never appears so that we can replace $\balpha$ by $\balpha^{i}$ which proves the proposition in the case of a positive edge. The proof in the case of a negative edge follows from $G=(G_{-e_{i}})_{-e_{i}}$.
\end{proof}

\begin{cor}[Change of the sign function]\label{cor:ChangeSignFct}
  Let $G_\veps$ be a ribbon graph with the sign function $\veps$ and $G_{-\veps}$ be the same ribbon graph only with the sign function $-\veps$. Then
  \begin{align}
    Z(G_{-\veps};q,\mathbf{\alpha},c)=&\Big(\prod_{e\in E(G)}\frac{\alpha_{e}}{\sqrt q}\Big)\,Z(G_{\veps};q,\frac{q}{\mathbf{\alpha}},c).
  \end{align}
\end{cor}
It is simply the proposition \ref{prop:FlipSign} applied to all the edges of $G_{\veps}$.

\begin{prop}[Natural duality]
  Let $G^{\star}$ be the natural dual of a ribbon graph $G$. Then
  \begin{align}
    Z(G;1,\mathbf{\alpha},c)=&\Big(\prod_{e\in E(G)}\alpha_{e}\Big)\,Z(G^{\star};1,\mathbf{\alpha}^{-1},c).
  \end{align}
\end{prop}
It is a direct consequence of theorem \ref{thm:PartialDuality} and corollary \ref{cor:ChangeSignFct}.

\subsection{Duality for the multivariate Tutte polynomial}
\label{sec:dual-mult-tutte}

In \cite{Chmutov2007aa} it has been shown that the duality relation for the Tutte polynomial of connected planar graphs is a special case of the partial duality for the signed \BRp. Here, we prove that the same result holds in the multivariate case.\\

\noindent
Let us first recall that the multivariate \Tp{} is defined as follows:
\begin{defn}[Multivariate \Tp{} \cite{Traldi1989aa}]
  \begin{align}
    Z_{T}(G;q,\balpha)\defi&\sum_{F\subseteq G}q^{k(F)}\prod_{e\in E(F)}\alpha_{e}.\label{eq:MultiTutteDef}
  \end{align}
\end{defn}
It obeys the following duality relation:
\begin{prop}[Duality for the multivariate \Tp{} \cite{Sokal2005aa}]\label{prop:MultiTutteDuality}
  Let $G$ be a connected planar graph and $G^{\star}$ its dual. The following relation holds:
  \begin{align}
    Z_{T}(G;q,\balpha)=&q^{1-v(G^{\star})}\Big(\prod_{e\in E(G)}\alpha_{e}\Big)Z_{T}(G^{\star};q,q/\balpha).\label{eq:MultiTutteDuality}
  \end{align}
\end{prop}

To derive the duality relation for the \Tp{}, S.~Chmutov \cite{Chmutov2007aa} used the fact that, for plane graphs, the \BRp{} $R(G;x,y,z)$ is independent of $z$ and reduces to the Tutte polynomial. We would like to maintain such features for the multivariate versions.

Let us recall that
\begin{align}
  Z(G;q,\balpha,c)=&\sum_{F\subseteq G}q^{k(F)+s(F)}\Big(\prod_{e\in E_{+}(F)}\alpha_{e}\prod_{e\in E_{-}(\bar{F})}\alpha_{e}\Big)c^{f(F)}.
\end{align}
If $c=1$ and if all the edges are positive, $Z(G;q,\alpha,c)=Z_{T}(G;q,\balpha)$, but this is clearly not the case for any $c$ when $G$ is plane. We have to define another multivariate version of the \BRp. To this aim, let us come back to the original polynomial:
\begin{subequations}
  \begin{align}
    R(G;x+1,y,z)=&\sum_{F\subseteq G}x^{r(G)-r(F)+s(F)}y^{n(F)-s(F)}z^{k(F)-f(F)+n(F)}\\
    =&x^{-k(G)}(yz)^{-v(G)}\sum_{F\subseteq G}(xyz^{2})^{k(F)+s(F)}(yz)^{e(F)-2s(F)}z^{-f(F)}\\
    =&x^{-k(G)}y^{-v(G)}\sum_{F\subseteq G}(xy)^{k(F)+s(F)}y^{e(F)-2s(F)}z^{2g(F)}.
  \end{align}
\end{subequations}
This shows that $z^{-v(G)}\sum_{F\subseteq G}(xyz^{2})^{k(F)+s(F)}(yz)^{e(F)-2s(F)}z^{-f(F)}$ is independant of $z$ if $G$ is plane. We propose then the following definition:
\begin{defn}[Signed multivariate \BRp{} $2$]
  \begin{align}
    Z_{R}(G;q,\balpha,z)\defi&z^{-v(G)}\sum_{F\subseteq G}(qz^{2})^{k(F)+s(F)}\Big(\prod_{e\in E_{+}(F)}z\alpha_{e}\prod_{e\in E_{-}(\bar{F})}z\alpha_{e}\Big)z^{-f(F)}\label{eq:SMBRPoly2Def}
  \end{align}
\end{defn}
We now give some properties of $Z_{R}$:
\begin{prop}
  Let $G$ be a ribbon graph.
  \begin{enumerate}
    \begin{subequations}
    \item $Z$ and $Z_{R}$ are related by:
      \begin{align}
        Z_{R}(G;q,\balpha,z)=&z^{-v(G)}Z(G;qz^{2},z\balpha,z^{-1}).\label{eq:ZZRRelation}
      \end{align}
    \item If $G$ is plane, $Z_{R}$ is independent of $z$.
    \item For any $G$ the edges of which are all positive, $Z_{R}(G;q,\balpha,1)=Z_{T}(G;q,\balpha)$.
    \item Under a flip of the signs, $Z_{R}$ transforms as follows:
      \begin{align}
        Z_{R}(G_{-\veps};qz^{2},z\balpha,z)=&\Big(\prod_{e\in E(G)}\frac{z\alpha_{e}}{\sqrt{qz^{2}}}\Big)Z_{R}(G_{\veps};qz^{2},zq/\balpha,z).\label{eq:ZRFlipSign}
      \end{align}
    \item Under the partial duality, $Z_{R}$ transforms as follows:
      \begin{align}
        \left.z^{v(G)}Z_{R}(G;qz^{2},z\balpha,z)\rabs_{qz^{2}=1}=&\left.z^{v(G')}Z_{R}(G';qz^{2},z\balpha,z)\rabs_{qz^{2}=1}.\label{eq:ZRGenDual}
      \end{align}
    \item Under the natural duality (duality with respect to $E(G)$ plus a flip of the signs), $Z_{R}$ transforms as follows:
      \begin{align}
        \left.z^{v(G)}Z_{R}(G;qz^{2},z\balpha,z)\rabs_{qz^{2}=1}=&z^{v(G^{\star})}\Big(\prod_{e\in E(G)}z\alpha_{e}\Big)Z_{R}(G^{\star};qz^{2},zq/\balpha,z)\big\vert_{qz^{2}=1}.\label{eq:ZRNatDual}
      \end{align}
    \end{subequations}
  \end{enumerate}
\end{prop}
\begin{proof}
  It is a simple application of the results derived in sections \ref{sec:generalised-duality-1} and \ref{sec:natural-duality-brp}.
\end{proof}

We can now prove that the duality relation for the multivariate \Tp{} (recalled in (\ref{eq:MultiTutteDuality})) is a direct consequence of the partial duality for the multivariate signed \BRp{} $Z_{R}$.

Let $G$ be a plane graph all the edges of which are positive. Then, for all $z$ we have
\begin{align}
  Z_{T}(G;q,\balpha)=&Z_{R}(G;qz^{2},z\balpha,z)=\left.Z_{R}(G;qz'^{2},z'\balpha,z')\rabs_{qz'^{2}=1}\nonumber\\
  =&z'^{-v(G)+v(G^{\star})}\Big(\prod_{e\in E(G)}z'\alpha_{e}\Big)Z_{R}(G^{\star};qz'^{2},z'q/\balpha,z')\big\vert_{qz'^{2}=1}.
    \intertext{Using $v(G)-e(G)+v(G^{\star})=2$ ($G$ being connected and plane), we get}
    Z_{T}(G;q,\balpha)=&(z'^{2})^{v(G^{\star})-1}\Big(\prod_{e\in E(G)}\alpha_{e}\Big)Z_{R}(G^{\star};qz'^{2},z'q/\balpha,z')\big\vert_{qz'^{2}=1}\nonumber\\
    =&q^{1-v(G^{\star})}\Big(\prod_{e\in E(G)}\alpha_{e}\Big)Z_{R}(G^{\star};qz^{2},zq/\balpha,z)\nonumber\\
    =&q^{1-v(G^{\star})}\Big(\prod_{e\in E(G)}\alpha_{e}\Big)Z_{T}(G^{\star};q,q/\balpha).
\end{align}

{\small
\bibliographystyle{fababbrvnat}
\bibliography{biblio-articles,biblio-books}
}
\end{document}